\documentclass[reqno,11pt, a4paper]{amsart}

\usepackage[T1]{fontenc}

\usepackage{amssymb}
\usepackage{amsmath}
\usepackage{amsfonts}
\usepackage{amsthm}
\usepackage{mathtools}
\usepackage{mathabx}
\usepackage{fullpage}
\usepackage{comment}
\usepackage[dvipsnames]{xcolor}
\usepackage{graphicx}
\usepackage{float}

\usepackage{mdframed}
\usepackage{enumitem} 

\usepackage[export]{adjustbox}
\usepackage{bbold}
\usepackage[percent]{overpic}

\usepackage{booktabs,array}
\usepackage{url} 
\usepackage[colorlinks=true,
linkcolor=black]{hyperref}

\usepackage{pgfplots}
\pgfplotsset{compat=1.6}

\usepackage{dashrule}

\usetikzlibrary{calc,quotes,angles,arrows.meta,positioning}
\usetikzlibrary{decorations.pathreplacing,decorations.markings}

\usepackage{float}

\usepackage{caption}
\usepackage{subcaption}

\usepackage[nocompress]{cite}

\usepackage{todonotes}

\theoremstyle{plain}%
\newtheorem{theorem}{Theorem}[section]
\newtheorem{lemma}[theorem]{Lemma}
\newtheorem{proposition}[theorem]{Proposition}
\newtheorem{corollary}[theorem]{Corollary}

\newtheorem*{conjecture*}{Conjecture}

 \numberwithin{equation}{section}

\theoremstyle{definition}

\theoremstyle{remark}
\newtheorem*{remark}{Remark}

\let \le \leqslant
 \let \leq \leqslant
 \let \geq \geqslant
 \let \ge \geqslant

\DeclareMathOperator{\area}{area}

\DeclareMathOperator{\Cov}{Cov}
\DeclareMathOperator{\unif}{unif}

\DeclareMathOperator{\support}{supp}

\DeclareMathOperator{\dist}{dist}

\usepackage{fancyhdr}
\usepackage{graphics}
\usepackage{graphicx}
\usepackage{soul}

\definecolor{detailcolor00}{rgb}{0.4405, 0.204, 0.343}
\definecolor{detailcolor01}{rgb}{0.546, 0.215, 0.352}
\definecolor{detailcolor02}{rgb}{0.675, 0.247, 0.387} 
\definecolor{detailcolor03}{rgb}{0.775, 0.317, 0.455}
\definecolor{detailcolor04}{rgb}{0.830, 0.421, 0.553} 
\definecolor{detailcolor05}{rgb}{0.831, 0.533, 0.663}
\definecolor{detailcolor06}{rgb}{0.779, 0.619, 0.775}
\definecolor{detailcolor07}{rgb}{0.724, 0.694, 0.827}
\definecolor{detailcolor08}{rgb}{0.687, 0.770, 0.880}
\definecolor{detailcolor09}{rgb}{0.671, 0.839, 0.904}
\definecolor{detailcolor10}{rgb}{0.659, 0.872, 0.882}

\usepackage{scalerel,stackengine}
\newcommand\pig[1]{\scalerel*[5.5pt]{\Big#1}{%
  \ensurestackMath{\addstackgap[1.5pt]{\big#1}}}}
\newcommand\pigl[1]{\mathopen{\pig{#1}}}
\newcommand\pigr[1]{\mathclose{\pig{#1}}}

\ifdim\overfullrule>0pt % draft option is active
  \usepackage{environ}

  \NewEnviron{tikzpicture}{%
    \begin{pgfpicture}
    \pgfpathrectanglecorners{\pgfpointorigin}{\pgfpoint{3cm}{3cm}}%
     \pgfusepath{stroke}\end{pgfpicture}%
  }
\fi

\makeatletter
\newcommand{\hathat}[1]{% 
\begingroup%
  \let\macc@kerna\z@%
  \let\macc@kernb\z@%
  \let\macc@nucleus\@empty%
  \hat{\raisebox{.3ex}{\vphantom{\ensuremath{#1}}}\smash{\hat{#1}}}%
\endgroup%
}

\newcommand{\smallhathat}[1]{% 
\begingroup%
  \let\macc@kerna\z@%
  \let\macc@kernb\z@%
  \let\macc@nucleus\@empty%
  \hat{\raisebox{.05ex}{\vphantom{\ensuremath{#1}}}\smash{\hat{#1}}}%
\endgroup%
}

\newcommand{\smallsmallhathat}[1]{% 
\begingroup%
  \let\macc@kerna\z@%
  \let\macc@kernb\z@%
  \let\macc@nucleus\@empty%
  \hat{\raisebox{-.2ex}{\vphantom{\ensuremath{#1}}}\smash{\hat{#1}}}%
\endgroup%
}
\makeatother

\makeatletter
\newcommand{\subalign}[1]{%
  \vcenter{%
    \Let@ \restore@math@cr \default@tag
    \baselineskip\fontdimen10 \scriptfont\tw@
    \advance\baselineskip\fontdimen12 \scriptfont\tw@
    \lineskip\thr@@\fontdimen8 \scriptfont\thr@@
    \lineskiplimit\lineskip
    \ialign{\hfil$\m@th\scriptstyle##$&$\m@th\scriptstyle{}##$\hfil\crcr
      #1\crcr
    }%
  }%
}
\makeatother

\title{Current expansion and couplings for \\Ising lattice gauge theory}

\author{Malin P. Forsstr\"om and Fredrik Viklund}
\address[Malin P. Forsstr\"om]{Mathematical Sciences, Chalmers University of Technology and University of Gothenburg, 412 96 Gothenburg, Sweden.}
\email{palo@chalmers.se}

\address[Fredrik Viklund]{Department of Mathematics, KTH Royal Institute of Technology, 100 44 Stockholm, Sweden.}
\email{frejo@kth.se}

\begin{document}

\maketitle 

\begin{abstract}
    In this note, we discuss a random current expansion and a switching lemma for Ising lattice gauge theory at all choices of inverse temperature $\beta$, leading to summation over surfaces. We also describe couplings of this expansion with other representations, including the high-temperature expansion and the  cluster expansion. We deduce some simple consequences, including several expressions for the Wilson loop expectation (at any $\beta$), a new proof of the area law estimate for sufficiently small \( \beta\), and a proof of exponential decay of correlations for small and large \( \beta. \) We also derive a few results analogous to corresponding results for the Ising model. In particular, we show that the Wilson loop expectation is non-negative at any $\beta$ and give an alternative short proof of Griffith's second inequality and, as a consequence, show that the Wilson loop expectations are increasing in \( \beta\) for all $\beta$. 
\end{abstract}

\section{Introduction}

Random current expansions for the Ising model were first introduced in~\cite{ghs1970}, and have since become an important tool in its analysis ~\cite{a1982, adcs2014, ADC, s1986,abf1987,dt2016,ghs1970}, as well as for the analysis of various related spin models such as the XY-model~\cite{p}.
Part of the reason for its importance is that it can be related to other useful representations, such as the FK-random cluster model~\cite{lw2016}. For reviews of the random current expansion and its applications to the classical Ising model, see~\cite{dc2017, HDC-Fields}.

The purpose of this note is to discuss a random current expansion for Ising lattice gauge theory on \( \mathbb{Z}^n,\) \( n \geq 3,\) along with some immediate consequences. In contrast to the classical Ising model, the current expansion here leads to summation over ``surfaces'' rather than lattice paths. We will also describe couplings with other graphical representations of the Ising lattice gauge theory. The proofs follow quite closely those of the corresponding results for the classical Ising model but we still feel these observations are worthwhile to record. 

\begin{remark}Aizenman's recent paper~\cite{a2025} independently describes the current expansion considered here, see Section~9 of \cite{a2025}. That paper was announced already in 1982 \cite{a1982} but appeared after the first version of this paper was posted. It does not discuss a switching lemma (for the gauge theory), couplings to other models, or the applications discussed in this paper. Instead, it utilizes a dual model which is related to the Ising model when the lattice is \( \mathbb{Z}^3.\)
\end{remark}

% In any case, the proofs closely follows that of the corresponding result for the classical Ising model but we still feel it is worthwhile to record, if only to highlight an interesting direction of future investigation. %For example, it would be very interesting to see if the current expansion can be used to analyze the the model near the confinement phase transition.

%A version of the switching lemma presented below was first used in~\cite{ghs1970}.

\subsection{Ising lattice gauge theory and Wilson loop expectations}\label{sec: preliminaries1}
Let \( m \geq 3 \). The lattice $\mathbb{Z}^m$ has a vertex at each point \( x \in \mathbb{Z}^m \) with integer coordinates and a non-oriented edge between each pair of nearest neighbors. To each non-oriented edge \( \bar e \) in \( \mathbb{Z}^m \) we associate two oriented edges \( e_1 \) and \( e_2 = -e_1 \) with the same endpoints as \( \bar e \) and opposite orientations.  

Let \( \mathbf{e}_1 \coloneqq (1,0,0,\ldots,0)\), \( \mathbf{e}_2 \coloneqq (0,1,0,\ldots, 0) \), \ldots, \( \mathbf{e}_m \coloneqq (0,\ldots,0,1) \) be oriented edges corresponding to the unit vectors in \( \mathbb{Z}^m \).

If \( v \in \mathbb{Z}^m \) and \( j_1 <   j_2 \), then \( p = (v +  \mathbf{e}_{j_1}) \land  (v+ \mathbf{e}_{j_2}) \) is a positively oriented 2-cell, also known as a positively oriented plaquette. 
For \( N \geq 1, \) we let \( B_N \) denote the set \(   [-N,N]^m \cap \mathbb{Z}^m \), and  let \( C_0(B_N) \), \( C_1(B_N) \), and \( C_2(B_N) \) denote the sets of oriented vertices, edges, and plaquettes, respectively, whose vertices are contained in \( B_N \). We write $C_1^+$ for positively oriented edges, and similarly for plaquettes.

Let $\mathbb{Z}_2$ be the additive group of integer modulo $2$. We let \( \Omega^1(B_N,\mathbb{Z}_2) \) denote the set of all  \( \mathbb{Z}_2 \)-valued  1-forms \( \sigma \) on \( C_1(B_N) \), i.e., the set of all \( \mathbb{Z}_2 \)-valued functions \(\sigma \colon  e \mapsto \sigma_e \) on \( C_1(B_N) \) such that \( \sigma_e =  -\sigma_{-e} \) for all \( e \in C_1(B_N) \). We write $\rho: \mathbb{Z}_2 \to \mathbb{C},\, g \mapsto e^{\pi i g}$ for the natural representation of $\mathbb{Z}_2$.

For an edge $e$, let $\hat{\partial}e$ be the co-boundary of $e$, that is, the $2$-chain of oriented plaquettes whose boundary contains $e$.

When \( \sigma \in \Omega^1(B_N,\mathbb{Z}_2) \) and \( p \in C_2(B_N) \), we let \( \partial p \) denote the four edges in the oriented boundary of \( p \) and define
\begin{equation*}
    d\sigma (p) \coloneqq \sum_{e \in \partial p} \sigma(e).
\end{equation*} 
Elements \( \sigma \in \Omega^1(B_N,\mathbb{Z}_2)  \) are referred to as \emph{gauge field configurations}.

The \emph{Wilson action functional} for pure gauge theory is defined by (see, e.g.,~\cite{w1974})
\begin{equation*}
    S(\sigma) \coloneqq - \sum_{p \in C_2(B_N)}  \rho \bigl( d\sigma(p) \bigr), \quad \sigma \in \Omega^1(B_N,\mathbb{Z}_2) .
\end{equation*}
The Ising lattice gauge theory probability measure on gauge field configurations is defined by
\[
\mu_{\beta,N}(\sigma)  \coloneqq
    Z_{\beta,N}^{-1} e^{-\beta S(\sigma)} , \quad \sigma \in \Omega^1(B_N,\mathbb{Z}_2) .
\]
Here for $N \geq 1$, \[Z_{\beta,N} \coloneqq \sum_{\sigma \in  \Omega^1(B_N,\mathbb{Z}_2) } e^{-\beta S(\sigma)}\] is the partition function and while we only consider the probability measure for positive $\beta$, the partition function is defined for $\beta \in \mathbb{C}$ when $N < \infty$. For $\beta \ge  0$, the corresponding expectation is written $\mathbb{E}_{\beta,N}.$ Let $\gamma$ be a nearest neighbor loop on $\mathbb{Z}^m$ contained in $B_N$. Given $\sigma \in \Omega^1(B_N,\mathbb{Z}_2) $, the Wilson loop variable for Ising lattice gauge theory is defined by
\[
W_\gamma = \rho\bigl( \sigma(\gamma) \bigr) = \prod_{e \in \gamma} \rho \bigl(\sigma(e) \bigr).
\]
For \( \beta \geq 0 \), let \( \langle W_\gamma \rangle_{\beta} \) denote the infinite volume limit of its expected value:
\[
    \langle W_\gamma \rangle_{\beta} \coloneqq \lim_{N \to \infty} \mathbb{E}_{\beta,N}[W_\gamma].
\]
See, e.g., \cite{flv2020} for a proof of the existence of this limit.

To simplify notations in what follows, we define
\begin{align*}
    Z_{\beta,N}[\gamma] \coloneqq \sum_{\sigma \in \Omega^1(B_N,\mathbb{Z}_2)} W_\gamma(\sigma) e^{-\beta S(\sigma)}.
\end{align*}

\subsection{Current expansions and the switching lemma}

We now describe the current expansion that naturally generalizes the current expansion of the Ising model. See Section~\ref{section: discrete exterior calculus} for discrete exterior calculus notations.
First, let 
\[
    \mathcal{C} \coloneqq \bigl\{ n \in \Omega^2(B_N,\mathbb{Z}) \colon n(p) \geq 0 \text{ for all } p \in C_2(B_N)^+ \bigr\},
\]
where \( n(p) \) denotes the value of \( n \) at \( p.\) 
When \( \gamma \) is a loop, we let 
\begin{equation}\label{eq: Cgamma}
    \mathcal{C}_\gamma \coloneqq \Bigl\{ n \in \mathcal{C} \colon  \gamma(e)+  \sum_{p \in \support \hat \partial e}n(p)  \equiv 0 \mod 2\quad \forall e \in C_1(B_N) \Bigr\} \subseteq \mathcal{C}.
\end{equation}
We will call \( \mathcal{C}_\gamma\) the set of currents with source \( \gamma.\) Considering the support of $n$, elements in $\mathcal{C}_\gamma$ correspond to ``surfaces'' with boundary $\gamma$. 

For \( n \in \mathcal{C},\) we define the $\beta$-dependent weight
\begin{equation}\label{eq: weight}
    w(n) = w_\beta(n) \coloneqq \prod_{p \in C_2(B_N)^+} \frac{( 2\beta  )^{n(p)}}{n(p)!},
\end{equation}
where by convention $0!=1$.

\begin{theorem}[Current expansion]\label{proposition: random currents ALGT Z2}
    For any loop \( \gamma, \) we have
    \begin{equation}\label{eq: random currents ALGT}
        \begin{split}
        & Z_{\beta,N}[\gamma]
        %\sum_{\sigma \in \Omega^1(B_N,\mathbb{Z}_2)} \rho(\sigma(\gamma)) e^{ \beta \sum_{p \in C_2(B_N)} \rho(d\sigma(p))} 
        =
        \bigl|\Omega^1(B_N,\mathbb{Z}_n)\bigr|
        \sum_{n \in \mathcal{C}_\gamma} w(n) .
        \end{split}
    \end{equation}
\end{theorem} 

We note that the representation given by Theorem~\ref{proposition: random currents ALGT Z2} is an analog of the random current representation for the Ising model. However, an important difference is that here, the currents are functions defined on plaquettes, whose natural ``sources'' are edges, while in the original setting, the currents are functions on the set of edges that have vertices as sources. The current expansion for Ising gauge theory, therefore, leads to surface sums as opposed to sums over paths.

The switching lemma is an important tool when applying the classical Ising model current expansion. A natural analog of this lemma for random currents on plaquettes is given below.
\begin{lemma}[Switching lemma]\label{lemma: the switching lemma}
    Let \( F \colon \mathcal{C} \to \mathbb{C}\) and let \( \gamma_1\) and \( \gamma_2  \) be two loops, and let \( \gamma_1 + \gamma_2 \) denote their concatenation. Then 
    \begin{equation*}
        \sum_{\substack{{n}_1 \in \mathcal{C}_{\gamma_1} \\ {n}_2 \in \mathcal{C}_{\gamma_2}}  } F({n}_1+{n}_2) w({n}_1)w({n}_2)
        =
        \sum_{\substack{{n}_1 \in \mathcal{C}_{0} \\ {n}_2 \in \mathcal{C}_{\gamma_1 + \gamma_2}}  }  F({n}_1+{n}_2) w({n}_1)w({n}_2) \mathbb{1} \bigl(\exists q \in \mathcal{C}_{\gamma_2} \colon q \leq  {n}_1 + {n}_2 \bigr).
    \end{equation*} 
\end{lemma}

\begin{remark}
    The expansion described in Theorem~\ref{proposition: random currents ALGT Z2} can easily be put in a more general setting.
    \begin{enumerate}
        \item We can allow the coupling constant \( \beta\) to be different for different plaquettes by letting \begin{equation*}
            w(n) \coloneqq \prod_{p \in C_2(B_N)^+} \frac{( 2\beta_p  )^{n(p)}}{n(p)!}.
        \end{equation*}
        \item We can consider theories with external fields as follows. To each edge \( e \in C_1(B_N),\) attach a plaquette \( p_e \) with \( e \in \partial p\) and assign it coupling constant \( \beta_p = \kappa.\) 
        Let 
        \[ \tilde P_N = \{ p_e \colon e \in C_1(B_N) \} \quad \text{and} \quad  \tilde E_n = \{ e' \in \partial p_e\smallsetminus \{ e \},\, p_e \in \tilde P_N \},\] and let \( \tilde \Omega^2 \) denote the set of 2-forms supported on \( C_2(B_N) \cup  \tilde P_2.\)
        Then lattice gauge theory on \( \tilde \Omega^2\) conditioned to be zero for all \( e' \in \support \partial p_e \smallsetminus \{ e \} \) and \( e \in C_1(B_N)^+\) is equivalent to the so-called lattice Higgs model with action
        \begin{equation*}
            -\beta \sum_{p \in C_2(B_N)} \rho(d\sigma(p)) - \kappa  \sum_{e \in C_1(B_N)} \rho(d\sigma(p)).
        \end{equation*}
        Further, if we let
        \[
    \tilde {\mathcal{C}} \coloneqq \bigl\{ n \in \tilde \Omega^2 \colon n(p) \geq 0 \text{ for all } p \in C_2(B_N)^+ \bigr\},
\]
let  \( \tilde{\mathcal{C}}_\gamma\) be defined analogously as before, and let
        \begin{equation*}
            w(n) \coloneqq \prod_{p \in C_2(B_N)^+} \frac{( 2\beta  )^{n(p)}}{n(p)!} \prod_{e \in C_1(B_N)^+} \frac{( 2\kappa  )^{n(p_e)}}{n(p_e)!}.
        \end{equation*} 
        then we get a current expansion and corresponding probability measure.
    \end{enumerate}
    In both of the above-described settings, our proofs of Theorem~\ref{proposition: random currents ALGT Z2} and Lemma~\ref{lemma: the switching lemma} extend immediately, and in addition, the switching lemma below also holds with the exact same proof.
\end{remark}

\begin{remark}
    In~\eqref{eq: Cgamma}, currents were defined as 2-forms. However, with small modifications, currents can instead be defined as functions \( n \colon C_2(B_N) \to \mathbb{Z}_+\) in analogy with the extension of random currents to the XY-model~\cite{p}.
\end{remark}

\subsection{Coupling with graphical expansions}

Couplings between various graphical expansions, such as the high-temperature expansion, current expansions, and the random cluster model, have been important tools for understanding the fundamental properties of Ising models and other models of statistical mechanics.

The \emph{high-temperature expansion} of \( Z_{N,\beta}[\gamma] \) is the following identity.
\begin{equation}
        \begin{split}
            & Z_{N,\beta}[\gamma] 
            =
            |\Omega^1(B_N,\mathbb{Z}_2)| (\cosh 2\beta)^{|C_2(B_N)^+|}
            \sum_{\substack{\omega \in \Omega^2(B_N,\mathbb{Z}_2) \colon\\ \delta \omega = \gamma}}(\tanh 2\beta)^{|(\support \omega)^+|}.
        \end{split}
    \end{equation} 
    For a proof, see, e.g., \cite{mmm1979}. 
    Letting 
\begin{equation}
	\mathcal{P}_\gamma^{ht} \coloneqq \bigl\{ \omega \in  \Omega^2(B_N,\mathbb{Z}_2) \colon \delta \omega = \gamma \bigr\},
\end{equation}
we define a probability measure, referred to as the \emph{high-temperature model} with boundary \( \gamma, \) by
\[ {P}^\gamma_{B_N,\beta} (\omega) \coloneqq \frac{ (\tanh 2\beta)^{|(\support \omega)^+|} }{\sum_{\omega' \in \mathcal{P}_\gamma^{(ht)}}(\tanh 2\beta)^{|(\support \omega)^+|}.} , \quad \omega \in  \mathcal{P}_\gamma^{ht}.
\]
The high-temperature model has been used in several recent papers, see, e.g., \cite{f2024,flv2025}, and is useful for instance as it maps the high-temperature regime of Ising lattice gauge theory to the low-temperature regime of a related model.

  Letting
	\[
	 \mathcal{P}_\gamma \coloneqq \bigl\{ P \subseteq C_2(B_N)^+ \colon d\sigma(P) \equiv 0 \Rightarrow \sigma(\gamma)=0 \bigr\}
	 =
	 \{ P \subseteq C_2(B_N)^+ \colon \exists P' \subset P  \text{ s.t. } \partial P' = \gamma 
	 \}
	 ,
	 \]
	 the \emph{random cluster expansion} is the identity
	\begin{align*} 
		&Z[\gamma] 
		= 
		\prod_{p \in C_2(B_N)^+} e^{2\beta_p} \sum_{P \in \mathcal{P}_\gamma }  2^{\mathbf{b}_1(P,\mathbb{Z}_2)}
		\prod_{p \in P} (1-e^{-4\beta_p})  \prod_{p\notin P}  e^{-4\beta_p},
	\end{align*} 
	where \( \mathbf{b}_1(P,\mathbb{Z}_2) \) the first Betti number, defined, e.g., by
\begin{equation}\label{eq: betti}
 	2^{\mathbf{b}_1(P,\mathbb{Z}_2)} \propto \sum_{\sigma \in \Omega^1(B_N,\mathbb{Z}_2)} \mathbf{1}(d\sigma(P) \equiv 0), \quad P \in \mathcal{P}_\gamma.
\end{equation}
See, e.g.~\cite{ds2023}. Here the proportionality constant depends on \( m, \) \( N \) and \( \beta \) only. With this in mind, for a loop \( \gamma, \) and \( p \in (0,1), \) we may define a probability measure
\begin{equation}
		\phi_{B_N,p}^\gamma (P) \coloneqq 
		\frac{ 2^{\mathbf{b}_1(P,\mathbb{Z}_2)} p^{|P|}   (1-p)^{|P^c|} 
		}{
			\sum_{P' \in \mathcal{P}_\gamma}  2^{\mathbf{b}_1(P',\mathbb{Z}_2)}   p^{|P'|}   (1-p)^{|(P')^c|}
		},\quad P \in 
		\mathcal{P}_\gamma,
	\end{equation}	
	referred to as the \emph{random cluster model} with boundary \( \gamma. \)
Since \( \mathcal{P}_0 = C_2(B_N)^+, \) this in particular implies that with \( p = 1-e^{4\beta}, \) we have
\begin{equation*}
	\mathbb{E}_{B_N,\beta}[W_\gamma] = \frac{Z[\gamma]}{Z[0]} = 
	\phi^0_{B_N,1-e^{-4\beta}}(\mathcal{P}_{\gamma}).
\end{equation*}
In particular, we note that, for this graphical representations the Wilson loop observable can be directly expressed as the probability of an event.
The random cluster model for lattice gauge theories was first introduced in~\cite[Section 5]{HS16} and then used in~\cite{ds2023,ds2024} to show that there is a sharp phase transition between perimeter and area law in \( 3 \)-dimensional Potts lattice gauge theory. 
The measure \( \phi_{B_N,\beta}^0 \) is known to have a number of useful properties, analogous to those of the ordinary random cluster model. In particular, by~\cite[Theorem 5.1]{HS16}, \( \phi_{B_N,\beta}^0 \)  satisfies the FKG lattice condition and is thus positively associated and satisfies the FKG inequality. The proof of this theorem generalizes verbatim to \( \phi_{B_N,(\beta_p)}^\gamma \), and hence this measure is also positively associated.  Moreover, by~\cite{HS16,ds2023} there is a natural analog of the Swendsen-Wang dynamics for this random cluster model (see Figure~\ref{figure: graph}).

   Finally, with~\eqref{eq: random currents ALGT} in mind, we define a probability measure associated with the random current expansion, referred to as the \emph{random current model with boundary} \( \gamma\),  by 
\begin{align}
	\mathbf{P}^\gamma_{B_N,\beta}({n}) \coloneqq \frac{\omega({n}) }{\sum_{{m}\in \mathcal{C}_\gamma} w({m})} ,\quad {n} \in \mathcal{C}_\gamma,
\end{align}
where we recall the definition of \( w(n) \) from~\eqref{eq: weight} and the definition of~\( \mathcal{C}_\gamma \) from~\eqref{eq: Cgamma}. 
When \( {n} \sim \mathbf{P}^\gamma_{B_N,\beta}, \) we can define an associated percolation model by letting \(  \hat {{n}} \coloneqq \mathbf{1}({n} >0), \) and let the corresponding probability measure of \( \hat n \) be denoted by \( \hat {\mathbf{P}}^\gamma_{B_N,\beta}. \)

We now present our next main result, which describes a coupling between the random current model, the random cluster model, and the high-temperature model. In this result and throughout the rest of this paper, for \( p \in [0,1] \) we let \( \Psi_p \) denote the law of iid \emph{Bernoulli plaquette percolation} with parameter \( p. \)

\begin{theorem}\label{proposition: couplings of many models}
	Let \( \beta > 0, \) and let \( \gamma \) be a loop. Let \( {n}\sim \mathbf{P}^\gamma_{B_N,\beta} \) (the random current model), \( \eta \sim {P}^\gamma_{B_N,\beta} \) (the high temperature expansion model), and let \( P \sim \phi_{B_N,\beta}^\gamma  \) (the random cluster model). 
	Further, let and \( X_1 \sim \Psi_{1 - 1/\cosh (2\beta)}\),   \( X_2 \sim \Psi_{\tanh 2\beta}\) and \( X_3 \sim \Psi_{1-e^{-2\beta}} \) be independent. Then the following holds.
	\begin{enumerate}[label=(\alph*)]
		\item \({n} \mod 2 \sim P^\gamma_{B_N,\beta} \) \label{item: 0 of coupling}
		\item  \( \hat {{n}} \coloneqq  \max(\eta,X_1) \sim \hat {\mathbf{P}}^\gamma_{B_N,\beta}. \)\label{item: 1 of coupling}
		\item If \( \omega \coloneqq  \max (\eta ,X_2) \overset{d}{=} \max(\hat{{n}},X_3), \) then \(   \support \omega \sim \phi_{B_N,\beta}^\gamma . \) \label{item: 2 of coupling}
		\item If \( P' \sim \unif \{  P' \subseteq P \colon  \partial P' = \gamma \} \) and \( \eta' \in \Omega^2(B_N,\mathbb{Z}_2) \) is defined by \( \support \eta' = P', \) then~\( \eta' \sim P_{B_N,\beta}^\gamma. \)   \label{item: 3 of coupling}
	\end{enumerate}
\end{theorem}

We illustrate the couplings of Theorem~\ref{proposition: couplings of many models} in~Figure~\ref{figure: graph}, also including the Swendsen-Wang dynamics described in~\cite{HS16,ds2023}.

\begin{remark}
	Note that Theorem~\ref{proposition: couplings of many models}\ref{item: 2 of coupling} immediately implies that \( \phi_{B_N,\beta} \) stochastically dominates an iid Bernoulli percolation process with parameter \( p_2 = \tanh 2\beta. \) We eleborate more on this later in Proposition~\ref{prop: stoch dom ii}.
\end{remark}

\begin{figure}[h]
	\begin{tikzpicture}
		\node[draw=red, align=center] (a) at (0,0) {\footnotesize High temperature expansion \\ \footnotesize \( \eta \sim P_{B_N,\beta}^\gamma\)};
		
		\node[draw=red, align=center] (b) at (0,4) {\footnotesize Random current model \\ \footnotesize \( {n} \sim {\mathbf{P}}^\gamma_{B_N,\beta}\)};
		
		\node[draw=red, align=center] (c) at (10,4) {\footnotesize Random current percolation model \\ \footnotesize \( \hat{{n}} \sim \hat{\mathbf{P}}^\gamma_{B_N,\beta}\)};
		
		\node[draw=red, align=center] (d) at (10,0) {\footnotesize Random cluster model \( P = \support \omega \sim \phi_{B_N,\beta}^\gamma \)};
		
		\node[draw=red, align=center] (d') at (10,-.8) {\footnotesize Random cluster model  \( P = \support \omega \sim \phi_{B_N,\beta}^0 \)};
		
		\node[draw=red, align=center] (e) at (10,-4) {\footnotesize Ising lattice gauge theory \( \sigma \sim \mu_{B_N,\beta}^0 \)};
		 
		 \draw[->] (b) -- (c) node[midway, above] {\footnotesize $\mathbf{1}({n}\geq 1)$};
		 
		 \draw[->] (a) -- (c) node[midway, anchor=south east] {\footnotesize $\max(\eta,X_1)$};
		 
		 \draw[->] (c) -- (d) node[midway, left] {\footnotesize $\max(\hat{{n}},X_3)$};

		 	\draw[->] (2.15,0.1) -- (6.74,0.1) node[midway, above] {\footnotesize $\max(\eta,X_2)$}; 
 
		 \draw[<-] (2.15,-0.1) -- (6.74,-0.1) node[midway, below] {\footnotesize $\unif(P' \subseteq P \colon \partial P' = \gamma)$};
		 
		 \draw[<-] (10.2,-1.15) -- (10.2,-3.68) node[midway, right] {\footnotesize $\delta_{1-e^{-2\beta}} \mathbf{1}(d\sigma(p)=0) $};
		 
		 \draw[<-] (e) -- (d') node[midway, left] {\footnotesize $ \unif(\sigma \in \Omega^{1}(B_N,\mathbb{Z}_2) \colon d\sigma(P) \equiv 0)  $};
	\end{tikzpicture}
	\caption{Above, we draw a graph summarizing the couplings between Ising lattice gauge theory, the random current model, the high-temperature expansion, and the random cluster model.}\label{figure: graph}
\end{figure}

\begin{remark}
	For the Ising model, the connection between the random cluster model and the Ising random current model was obtained in~\cite{lw2016} and the connection between the high temperature expansion and the random cluster model in~\cite{gs2008}.
For lattice gauge theories, in the case \( \gamma=0,\) the coupling between \( \phi^0_{B_N,\beta} \) and \( \mu^0_{B_N,\beta} \) first appeared in~\cite{ds2023} and~\cite{HS16}, and the coupling between \( P^0_{B_N,\beta} \) and \( \phi^0_{B_N,\beta} \) first appeared as~\cite[Proposition 4.7]{hjk2025}. However, we stress that since the measures \( P^\gamma_{B_N,\beta}, \) \( \phi^\gamma_{B_N,\beta}, \) and \( \hat{\mathbf{P}}^\gamma_{B_N,\beta}, \) are the measures that naturally appear from taking the corresponding expectations of a Wilson loop observable, not requiring \( \gamma = 0 \) might prove useful.
\end{remark}

We present several easy applications of our main results in Section~\ref{section: applications}, including a double-current expansion of Wilson loop observables, a proof of positivity of Wilson loop observables, a proof of Griffith's second inequality, a proof for area law for Wilson loop observables for small \( \beta,\) and exponential decay of correlations for small and large \( \beta>0. \)

\subsection*{Acknowledgements}
F.V. acknowledges support by the Knut and Alice Wallenberg Foundation and the Gustafsson Foundation KVA. M.P.F. acknowledges support from the Swedish Research council, grant number 2024-04744. We thank Juhan Aru, Paul Duncan, and Benjamin Schweinhart for discussions.

\section{Notation}\label{section: discrete exterior calculus}

Since several recent papers contain thorough introductions to the notations of discrete exterior calculus, we keep this section short, and refer the reader to~\cite{flv2020} for further details.

\begin{itemize}
%\item  We work with the square lattice $\mathbb{Z}^m$, where we assume that the dimension $m \ge 3$ throughout. We write $B_N = [-N,N]^m \cap \mathbb{Z}^m$.
\item{We write $C_k(B_N)$ and $C_k(B_N)^+$ for the set of unoriented and positively oriented $k$-cells, respectively (see \cite[Sect. 2.1.2]{flv2020}). An oriented $2$-cell is called a plaquette.}
\item{Formal sums of positively oriented \( k \)-cells with integer coefficients are called $k$-chains, and the space of $k$-chains is denoted by \( C_k(B_N,\mathbb{Z}) \), (see \cite[Sect. 2.1.2]{flv2020})}
\item{Let \( k \geq 2 \) and \( c = \frac{\partial}{\partial x^{j_1}}\big|_a \wedge \dots \wedge \frac{\partial}{\partial x^{j_k}}\big|_a \in C_k(B_N)\). The \emph{boundary} of $c$ is the $(k-1)$-chain \(\partial c \in C_{k-1}(B_N, \mathbb{Z})\) defined as the formal sum of the \( (k-1)\)-cells in the (oriented) boundary of \( c.\) 
The definition is extended to $k$-chains by linearity (see \cite[Sect.~2.1.4]{flv2020}).}
\item{If \( k \in \{ 0,1, \ldots, n-1 \} \) and \( c \in C_k(B_N)\) is an oriented \( k \)-cell, we define the \emph{coboundary} \( \hat \partial c \in C_{k+1}(B_N)\) of \( c \) as the \( (k+1) \)-chain $\hat \partial c \coloneqq \sum_{c' \in C_{k+1}(B_N)} \bigl(\partial c'[c] \bigr) c'.$ See \cite[Sect.~2.1.5]{flv2020}.}

\item{We let $\Omega^k(B_N, G)$ denote the set of $G$-valued (discrete differential) $k$-forms (see \cite[Sect 2.3.1]{flv2020}); the exterior derivative $d : \Omega^k(B_N, G)  \to \Omega^{k+1}(B_N, G)$ is defined for $0 \le k \le m-1$ (see \cite[Sect. 2.3.2]{flv2020}) and $\Omega^k_0(B_N, G)$ denotes the set of closed $k$-forms, i.e., $\omega \in \Omega^k(B_N, G)$ such that $d\omega = 0$.}
\item{We write $\support \omega = \{c \in C_k(B_N): \omega(c) \neq 0\}$ for the support of a $k$-form $\omega$. Similarly, we write $(\support \omega)^+ = \{c \in C_k(B_N)^+: \omega(c) \neq 0\}$}
\item{
    A 1-chain \( \gamma \in C_1(B_N,\mathbb{Z}) \) with finite support \( \support \gamma \) is called a \emph{loop} if 
    for all \( e \in \Omega^1(B_N) \), we have that \( \gamma[e] \in \{ -1,0,1 \} \), and \( \partial \gamma = 0. \) We write \( |\gamma| = |\support \gamma|.\) (In \cite{flv2020} this object was called a generalized loop.)
}

%\item  Let \( \gamma \in C_1(B_N,\mathbb{Z}) \) be a loop. A \( 2 \)-chain \( q \in C_2(B_N,\mathbb{Z}) \) is an \emph{oriented surface} with \emph{boundary} \( \gamma \) if \( \partial q = \gamma. \) 
\end{itemize}

\section{The current expansion}\label{section: current expansion}

In this section, we prove Theorem~\ref{proposition: random currents ALGT Z2}. For the proof, recall the definitions of \( \mathcal{C}_\gamma \) from~\eqref{eq: Cgamma} and the definition of the corresponding weight from~\eqref{eq: weight}. With this setup, the proof closely follows the corresponding proof for the Ising model, see, e.g., \cite{dc2017}.

\begin{proof}[Proof of Theorem~\ref{proposition: random currents ALGT Z2}] 
    For any \( \sigma \in \Omega^1(B_N,\mathbb{Z}_2), \) we have
    \begin{equation*}
        \begin{split}
            & e^{ \beta \sum_{p \in C_2(B_N)} \rho(d\sigma(p))}
            = \prod_{p \in C_2(B_N)^+} e^{ 2\beta   \rho(d\sigma(p))}
            = \prod_{p \in C_2(B_N)^+} \sum_{n(p) \geq 0} \frac{\bigl( 2\beta   \rho(d\sigma(p)) \bigr)^{n(p)}}{n(p)!}
            \\&\qquad= 
            \sum_{n \in \mathcal{C}} \prod_{p \in C_2(B_N)^+}  \frac{\bigl( 2\beta   \rho(d\sigma(p)) \bigr)^{n(p)}}{n(p)!}
            =
            \sum_{n \in \mathcal{C}} w(n) \prod_{p \in C_2(B_N)^+} \pigl(  \rho\bigl(d\sigma(p)\bigr) \pigr)^{n(p)}.
        \end{split}
    \end{equation*}
    Now fix   \( e \in C_1(B_N)^+,\) and let \( \sigma, \sigma' \in \Omega^1(B_N, \mathbb{Z}_2) \) be such that 
    \begin{enumerate}
        \item \( \sigma(e') = \sigma'(e')\) for all \( e' \in C_1(B_N)^+ \smallsetminus \{ e \},\) and
        \item \( \rho(\sigma(e)) =- \rho(\sigma'(e)).\)
    \end{enumerate}
    Then, for any \( n \in \mathcal{C},\) we have
    \begin{equation*}
        \begin{split}
            &
            \rho\bigl(\sigma(\gamma)\bigr)\prod_{p \in C_2(B_N)^+} \pigl(    \rho \bigl(d\sigma(p)\bigr) \pigr)^{n(p)} 
            +
            \rho\bigl(\sigma'(\gamma)\bigr)\prod_{p \in C_2(B_N)^+} \bigl(   \rho(d\sigma'(p)) \bigr)^{n(p)}  
            \\&\qquad=
            \rho\bigl(\sigma(\gamma)\bigr)\prod_{p \in C_2(B_N)^+} \bigl(  \rho(d\sigma(p)) \bigr)^{n(p)}
            \biggl( 1
            +
            (-1)^{\gamma(e)} \prod_{p \in \support \hat \partial e} (-1)^{n(p)} \biggr),
        \end{split}
    \end{equation*} 
    where 
    \begin{equation*}
        \begin{split}
            &
            1
            +
            (-1)^{\gamma(e)} \prod_{p \in \support \hat \partial e} (-1)^{n(p)} 
            =
            1
            +
            (-1)^{\gamma(e) + \sum_{p \in \support \hat \partial e} n(p)} 
            \\&\qquad=
            \begin{cases}
                2 &\text{if } \gamma(e) + \sum_{p \in \support \hat \partial e} n(p) \equiv 0 \mod 2\cr
                0 &\text{otherwise.}
            \end{cases}
        \end{split} 
    \end{equation*}  
    Since \( e \) was arbitrary, it follows that  
    \begin{align*}
        &\sum_{\sigma \in \Omega^1(B_N,\mathbb{Z}_2)}  \rho\bigl(\sigma(\gamma)\bigr)\prod_{p \in C_2(B_N)^+} \pigl(    \rho \bigl(d\sigma(p)\bigr) \pigr)^{n(p)} 
        \\&\qquad=
        \mathbb{1}(n \in \mathcal{C}_\gamma)\sum_{\sigma \in \Omega^1(B_N,\mathbb{Z}_2)}  \rho\bigl(\sigma(\gamma)\bigr)\prod_{p \in C_2(B_N)^+} \pigl(    \rho \bigl(d\sigma(p)\bigr) \pigr)^{n(p)} .
    \end{align*}
    Now note that, for \( n \in \mathcal{C_\gamma},\) we have
    \begin{align*}
        &\sum_{\sigma \in \Omega^1(B_N,\mathbb{Z}_2)}  \rho\bigl(\sigma(\gamma)\bigr)\prod_{p \in C_2(B_N)^+} \pigl(    \rho \bigl(d\sigma(p)\bigr) \pigr)^{n(p)} 
        \\&\qquad
        = 
        \sum_{\sigma \in \Omega^1(B_N,\mathbb{Z}_2)}  \prod_{e \in C_1(B_N)^+}   \pigl(    \rho \bigl(\sigma(e)\bigr) \pigr)^{\gamma(e) + \sum_{p \in \support \hat \partial e} n(p)}  
        = 
        \bigl| \Omega^1(B_N,\mathbb{Z}_2)\bigr|.
    \end{align*}
    Combining the above equations, we obtain the desired conclusion.
\end{proof}

 \section{The switching lemma}\label{section: switching lemma}

In this section, we give a proof of Lemma~\ref{lemma: the switching lemma}, which is generally known as the switching lemma. As for the proof of Theorem~\ref{proposition: random currents ALGT Z2}, the proof closely follows the corresponding proof of the Ising model, with smaller modifications which are due to the more complex geometry involved in the definition of our model.

\begin{proof}[Proof of Lemma~\ref{lemma: the switching lemma}]
    For \( {n}_1,{n}_2 \in \mathcal{C},\) let
    \begin{equation*}
        \binom{{n}_1+{n}_2}{{n}_1} \coloneqq \prod_{p\in C_2(B_N)^+} \binom{{n}_1(p)+{n}_2(p)}{{n}_1(p)}.
    \end{equation*} 
    Then for any \( {n}_1,{n}_2 \in \mathcal{C}, \)   we have
    \begin{equation}\label{eq: sum of w}
        \binom{{n}_1+{n}_2}{{n}_1} w({n}_1+{n}_2) = w({n}_1)w({n}_2).
    \end{equation}  
    Consequently, \begin{equation}\label{eq: switching lemma lhs}
        \begin{split}
            &\sum_{\substack{{n}_1 \in \mathcal{C}_{\gamma_1} \\ {n}_2 \in \mathcal{C}_{\gamma_2}}  } F({n}_1+{n}_2) w({n}_1)w({n}_2)
            =
            \sum_{{n} \in \mathcal{C}_{\gamma_1+ \gamma_2}}
            \sum_{\substack{{n}_1 \in \mathcal{C}_{\gamma_1} \\ {n}_2 \in \mathcal{C}_{\gamma_2}}  } F({n}_1+{n}_2) w({n}_1)w({n}_2) \cdot \mathbb{1}({n}_1 + {n}_2 = {n})
            \\&\qquad=
            \sum_{{n} \in \mathcal{C}_{\gamma_1+ \gamma_2}}
            F({n}) w({n})
            \sum_{\substack{{n}_1 \in \mathcal{C}_{\gamma_1} \mathrlap{\colon} \\ {n}_1 \leq {n} }  }  \binom{{n}}{{n}_1}.
        \end{split}
    \end{equation}
    Analogously,
    \begin{equation}\label{eq: switching lemma rhs}
        \begin{split}
            &
            \sum_{\substack{{n}_1 \in \mathcal{C}_{\gamma_1 + \gamma_2} \\ {n}_2 \in \mathcal{C}_{0}}  }  F({n}_1+{n}_2) w({n}_1)w({n}_2) \cdot \mathbb{1}\bigl(\exists q \in \mathcal{C}_{\gamma_1}  \colon q \leq {n}_1 + {n}_2  \bigr)
            \\&\qquad=
            \sum_{{n} \in \mathcal{C}_{\gamma_1+ \gamma_2}}
            \sum_{\substack{{n}_1 \in \mathcal{C}_{0} \\ {n}_2 \in \mathcal{C}_{\gamma_1 + \gamma_2}}  } F({n}_1+{n}_2) w({n}_1)w({n}_2) \cdot \mathbb{1}({n}_1 + {n}_2 = {n}) \cdot \mathbb{1} \bigl(\exists q \in \mathcal{C}_{\gamma_1} \colon q \leq {n}_1 + {n}_2  \bigr)
            \\&\qquad=
            \sum_{{n} \in \mathcal{C}_{\gamma_1+ \gamma_2}}
            F({n}) w({n}) \cdot \mathbb{1} \bigl(\exists q \in \mathcal{C}_{\gamma_1}  \colon  q \leq {n} \bigr)
            \sum_{\substack{{n}_1 \in \mathcal{C}_{0} \mathrlap{\colon} \\ {n}_1 \leq {n} }  }  \binom{{n}}{{n}_1}.
        \end{split}
    \end{equation}
    Combining~\eqref{eq: switching lemma lhs} and~\eqref{eq: switching lemma rhs}, the desired conclusion will follow if we can show that for each \( n \in \mathcal{C}_{\gamma_1 + \gamma_2},\) we have 
    \begin{equation}\label{eq: goal}
        \sum_{\substack{ {n}_1 \in \mathcal{C}_{\gamma_1} \mathrlap{\colon} \\  {n}_1 \leq  {n} }  }  \binom{ {n}}{ {n}_1} 
        = 
        \mathbb{1} \bigl(\exists q \in \mathcal{C}_{\gamma_1}  \colon {q} \leq {n}  \bigr)
        \sum_{\substack{{n}_1 \in \mathcal{C}_{0} \mathrlap{\colon} \\  {n}_1 \leq  {n} }  }  \binom{ {n}}{ {n}_1}.
    \end{equation} 
    To this end, fix some \( n \in \mathcal{C}_{\gamma_1 + \gamma_2}\)  and note first that if any of the two sides of~\eqref{eq: goal} is equal to zero, then \( \{ {q} \in \mathcal{C}_{\gamma_1} \colon {q} \leq {n} \} = \emptyset,\) and hence in this case both sides of~\eqref{eq: goal} are equal to zero. 

    Now assume that there exists \( {q} \in \mathcal{C}_{\gamma_1}\) such that \( {q} \leq n. \) Fix one such \( {q}.\)  
    For each \( {m} \leq {n},\) we can associate a set of sequences 
    \[
        \mathcal{S}_{{m}} \coloneqq 
        \Bigl\{ (S_p)_{p \in C_2(B_N)^+} \colon S_p \in \binom{\bigl\{1,2,\dots, {n}(p) \bigr\}}{{m}(p)} \quad \forall p \in C_2(B_N)^+   \Bigr\}.
    \]
    The mapping \( {m} \mapsto \mathcal{S}_{{m}}\) is clearly bijective.
    Moreover
    \begin{equation*}
        \sum_{\substack{ {n}_1 \in \mathcal{C}_{\gamma_1} \mathrlap{\colon} \\  {n}_1 \leq  {n} }  }  \binom{ {n}}{ {n}_1} 
        = 
        \sum_{\substack{ {n}_1 \in \mathcal{C}_{\gamma_1} \mathrlap{\colon} \\  {n}_1 \leq  {n} }  } \bigl|\mathcal{S}_{{n}_1} \bigr|
        =
        \biggl| \bigsqcup_{\substack{ {n}_1 \in \mathcal{C}_{\gamma_1} \mathrlap{\colon} \\  {n}_1 \leq  {n} }  } \mathcal{S}_{{n}_1} \biggr|.
    \end{equation*}
    
    Fix one element \( (S_p^{q}) \in \mathcal{S}_{q},\) and define the map \( \varphi \) by
    \begin{equation*}
        \varphi \colon (S_p)_{p \in C_p(B_N)^+} \mapsto (S_p \triangle S_p^{q})_{e \in C_2(B_N)^+}.
    \end{equation*}
    Then \( \varphi\) is a bijection between \( \bigsqcup_{ n_1 \in \mathcal{C}_{\gamma_1} \colon  n_1 \leq {n}} \mathcal{S}_{n_1}\) and \( \bigsqcup_{n_1 \in \mathcal{C}_{\gamma_0} \colon  n_1 \leq {n}} \mathcal{S}_{n_1},\) and hence the desired conclusion follows.
\end{proof}

\section{Couplings between graphical representations}

In this section, we give a proof of Theorem~\ref{proposition: couplings of many models}. We note that apart from a few topological considerations, the proof is almost identical to the analog proofs for the Ising model, as presented in, e.g., \cite[Section 3.2]{adtw2018}.

\begin{proof}[Proof of Theorem~\ref{proposition: couplings of many models}]

	To simplify notation, let \[ p_1 \coloneqq 1 - 1/\cosh (2\beta) , \quad  p_2 \coloneqq  \tanh 2\beta, \quad  p_3 \coloneqq  1-e^{-2\beta}.\]

	We first note that~\ref{item: 0 of coupling} is immediate from the definition of \( \hat{\mathbf{P}}^\gamma_{B_N,\beta}. \)

	We now show that~\ref{item: 1 of coupling} holds. To this end, for \( \eta \sim P^\gamma_{B_N,\beta} \) (the high temperature expansion) and \( X_2 \sim \Psi_{1-1/\cosh(2\beta)}, \)  let \( {\hat n  \coloneqq \max(\eta,X_1) . }\) We need to show that  \(   \hat n \sim \hat {\mathbf{P}}^\gamma_{B_N,\beta}. \) For this, note first that if \( {n} \sim \mathbf{P}_{B_N,\beta}^\gamma\), then \( \eta \overset{d}{=} {n} \mod 2.\) 
	From this it follows that
	\begin{align*}
		&\mathbf{P}_{B_N,\beta}^\gamma\bigl( \hat {{n}}(p)=  1 \mid (n \mod 2)(p) =0 \bigr)
		=
		\mathbf{P}_{B_N,\beta}^\gamma\bigl( {n}(p)\geq  1 \mid (n \mod 2)(p) =0 \bigr)
		\\&\qquad =
		\frac{\sum_{m\in \mathbb{Z} \colon m \text{ even} } \frac{(2\beta)^m}{m!} \mathbf{1}(m \geq 2)}{\sum_{m\in \mathbb{Z} \colon m \text{ even} } \frac{(2\beta)^m}{m!} } = \frac{\cosh 2 \beta -1}{\cosh 2\beta} = 1- 1/\cosh(2 \beta) 
		\\&\qquad = {P}_{B_N,\beta}^\gamma\times \Psi_{1-1/\cosh(2\beta)} \bigl(\max\bigl(\eta(p),X_1(p)\bigr)=1 \mid \eta(p) =0 \bigr).
	\end{align*}
	and
	\begin{align*}
		&\mathbf{P}_{B_N,\beta}^\gamma\bigl( \hat {{n}}(p)=  1 \mid (n \mod 2)(p) =1 \bigr)
		=
		1.
	\end{align*}
	This concludes the proof of~\ref{item: 1 of coupling}.

	We now proceed to the proof of~\ref{item: 2 of coupling}. To this end, let \( \eta \sim {P}^\gamma_{B_N,\beta} \) (the high temperature expansion), let \( \eta' \coloneqq \max(\eta,X_2) \), and consider the joint law~\( \mathbb{P} \) of the pair  \( ( \eta, \eta'). \) Then \( ( \eta, \eta') \) is supported on the space \( \Omega_\gamma \) of pairs \( (\eta,\eta') \) satisfying \( \eta \leq \eta' \) and \( \delta \eta = \gamma. \) Moreover, since \( \eta \) 
	satisfies \( \delta \eta = \gamma, \) we have \( (\support \eta)^+ \in \mathcal{P}_\gamma. \) Since \( \eta' \geq \eta, \) it follows that \( (\support \eta')^+  \in \mathcal{P}_\gamma. \)  
	Furthermore,
	\begin{align*}
		&\mathbb{P}\bigl((\eta,\eta')\bigr) 
		= 
		\frac{\mathbf{1}((\eta,\eta') \in \Omega_\gamma)}{Z_0} \cdot (\tanh 2\beta)^{|(\support \eta)^+|} p_2^{|(\support \eta')^+ \smallsetminus (\support \eta)^+|} (1-p_2)^{|C_2(B_N)^+\smallsetminus (\support \eta')^+|}
		\\&\qquad = 
		\frac{\mathbf{1}((\eta,\eta') \in \Omega_\gamma)}{Z_0} \cdot (\tanh 2\beta)^{|(\support \eta')^+|}   (1-\tanh 2\beta)^{|C_2(B_N)^+\smallsetminus (\support \eta')^+|}
		\\&\qquad = 
		\frac{\mathbf{1}((\eta,\eta') \in \Omega_\gamma)}{Z_0 (\frac{1+e^{-4\beta}}{2})^{|C_2(B_N)^+|}} \cdot \bigl(\frac{p_3}{2} \bigr)^{|(\support \eta')^+|}   (1-p_3)^{|C_2(B_N)^+\smallsetminus (\support \eta')^+|}
	\end{align*}
	where \( Z_0 \) is a normalizing constant.
	
	%\textcolor{purple}{[In the LGP-RClM coupling, we also restrict sprinkling to a subset.]}
	
	Now, define a probability measure \( \hat {\mathbb{P}} \) on  \( \Omega_\gamma \) as follows. Choose~\( P' \sim \phi_{B_N,\beta}^\gamma \) (the random cluster model ) and choose~\( P\) uniformly in the set 
	\[ S_\gamma(P') \coloneqq \{  P \subseteq P' \colon  \partial P = \gamma \}. 
	\] 
	%\textcolor{purple}{In the LGP-RClM coupling, we choose the subset uniformly in \( Z^{i-1}(P,\mathbb{Z}_2) , \) meaning the set of 1-forms with \( d\sigma(P) \equiv 0. \)] }
	Since \( P' \in \mathcal{P}_{\gamma} , \)  the set \( S_\gamma(P') \) is non-empty, implying in particular that \( P \) is well defined. Letting \( P_0 \in S_\gamma(P'), \) the map \( P' \Delta P_0 \), where \( \Delta\) denotes symmetric difference, is a bijection from \( S_\gamma(P') \) to \( S_0(P'), \) and hence \( |S_\gamma(P')| = |S_0(P')|. \)
	Let \( \eta_P \) and \( \eta_{P'} \) be the  2-forms corresponding to \( P \) and \( P' \) in the sense that \( (\support \eta_P)^+ = P \) and \( {(\support \eta_{P'})^+ = P'.} \) Then, by construction, we have \( (\eta_P,\eta_{P'}) \in \Omega_\gamma. \)
	Since \( P\) is chosen uniformly in \( S_\gamma( P'), \) we deduce that
	\begin{align*}
		&\hat {\mathbb{P}} \bigl((P, P' )\bigr) 
		= 
		\frac{\mathbb{1}((\eta_P, \eta_{P'})\in \Omega_\gamma)}{Z_1} \Bigl[ \sum_{\sigma \in \Omega_1(B_N,\mathbb{Z}_2)} \!\!\!\!\!\!\!\!\mathbf{1}(d\sigma(P') \equiv 0)\Bigr]  p^{|P'|}   (1-p)^{|C_2(B_N)^+\smallsetminus P'|} \cdot |S_\gamma(P')|^{-1}
		\\&\qquad= 
		\frac{\mathbb{1}((\eta_P, \eta_{P'})\in \Omega_\gamma)}{Z_1}  \bigl(\frac{p_3}{2}\bigr)^{|P'|}  (1-p_3)^{|C_2(B_N)^+\smallsetminus P'|} \cdot \frac{2^{|P'|} \sum_{\sigma \in \Omega_1(B_N,\mathbb{Z}_2)} \mathbf{1}(d\sigma(P') \equiv 0) }{|S_0(P')|}  .
	\end{align*} 
	Since 
	\begin{align*}
		\sum_{\sigma \in \Omega_1(B_N,\mathbb{Z}_2)} \!\!\!\!\!\!\!\! \mathbf{1}(d\sigma(P') \equiv 0) = 2^{\mathbf{b}_1(P',\mathbb{Z}_2)} = 2^{|C_1(B_N)^+| - |P'|} |S_0(P')|,
	\end{align*}
	it follows that the measures \( \mathbb{P} \) and \( \hat {\mathbb{P}} \) are equal. Since the  first marginal of \( \hat{\mathbb{P}} \) has law \( P^\gamma_{B_N,\beta} \) and the second marginal of \( \hat{\mathbb{P}} \) has law \( \phi_{\Lambda_2,\beta}^\gamma, \) this completes the proofs of~\ref{item: 2 of coupling} and~\ref{item: 3 of coupling}.  
\end{proof}

\section{Applications}\label{section: applications}

We now discuss a few results that are well-known but easy to prove using the current expansion.

Let us first, however, note the following representation of the square of the Wilson loop expectation current expansion, in analogy with the corresponding expansions for spin-spin correlations in the Ising model.
\begin{proposition}\label{proposition: measure}
    Let \( \beta \geq 0 \) and \( N \geq 1,\) and let \( \gamma \) be a loop. Then 
    \begin{equation*}
        \mathbb{E}_{N,\beta}\left[W_\gamma \right]^2 
        =  
        \frac{\sum_{{n}_1, {n}_2 \in \mathcal{C}_{0}   } w({n}_1)w({n}_2) \mathbb{1} \bigl(\exists q \in \mathcal{C}_{\gamma} \colon q \leq  {n}_1 + {n}_2   \bigr)}{\sum_{{n}_1, n_2 \in \mathcal{C}_{0} } w({n}_1)w({n}_2)}.
    \end{equation*}
\end{proposition}

\begin{proof}
    Using Lemma~\ref{lemma: the switching lemma} with \( F = 1,\)  \( \gamma_1= \gamma,\) and \( \gamma_2 = -\gamma,\) we obtain
    \begin{align*}
        &\mathbb{E}_{N,\beta}\left[W_\gamma \right]^2 
        = 
        \mathbb{E}_{N,\beta}\left[W_\gamma \right]\mathbb{E}_{N,\beta}\left[W_{-\gamma} \right]
        = 
        \frac{\sum_{{n}_1 ,{n}_2\in \mathcal{C}_{\gamma}   } w({n}_1)w({n}_2)}{\sum_{{n}_1 ,{n}_2 \in \mathcal{C}_{0}  } w({n}_1)w({n}_2)}
        \\&\qquad= 
        \frac{\sum_{{n}_1, {n}_2 \in \mathcal{C}_{0}   } w({n}_1)w({n}_2) \mathbb{1} \bigl(\exists q \in \mathcal{C}_{\gamma}  \colon q \leq  {n}_1 + {n}_2 \bigr)}{\sum_{{n}_1, n_2 \in \mathcal{C}_{0} } w({n}_1)w({n}_2)}.
    \end{align*}
    This completes the proof.
\end{proof}

The following fact was pointed out in \cite{Chatterjee18} and another proof was given in \cite{FV}.
\begin{proposition}\label{proposition: positivity}
    Let \( \beta \geq 0 \) and \( N \geq 1,\) and let \( \gamma \) be a loop. Then 
    \[
    \mathbb{E}_{N,\beta}[W_\gamma] > 0.
    \] 
\end{proposition}

\begin{proof}
    The desired conclusion follows immediately from Theorem~\ref{proposition: random currents ALGT Z2}. 
\end{proof}

We next show how the current expansion can be used to give a short proof of Griffith's second inequality, and then show how this results imply that the Wilson loop expectation is increasing in \( \beta.\)
\begin{proposition}\label{proposition: griffiths}
    Let \( \beta \geq 0 \) and \( N \geq 1,\) and let \( \gamma, \) \( \gamma_1, \) and \( \gamma_2\) be loops. Then
    \begin{enumerate}[label=(\roman*)]
        \item \(  \mathbb{E}_{N,\beta}  \left[W_{\gamma_1} W_{\gamma_2} \right] \geq \mathbb{E}_{N,\beta}\left[W_{\gamma_1} \right]\mathbb{E}_{N,\beta} \left[W_{\gamma_2} \right] \) and\label{item: griffiths}
        \item \( \frac{d}{d\beta} \mathbb{E}_{N,\beta} \left[ W_\gamma \right] \geq 0. \) \label{item: monotonicity}
    \end{enumerate}
\end{proposition}

\begin{proof}
    Using Lemma~\ref{lemma: the switching lemma} (the switching lemma) with \( F = 1, \) we can write
    \begin{align*}
        &\frac{\mathbb{E}_{N,\beta}\left[W_{\gamma_1} \right] \mathbb{E}_{N,\beta}\left[W_{\gamma_2} \right]}{\mathbb{E}_{N,\beta}\left[W_{\gamma_1} W_{\gamma_2} \right]}
        =
        \frac{ \mathbb{E}_{N,\beta}\left[W_{\gamma_1} \right] \mathbb{E}_{N,\beta}\left[W_{\gamma_2} \right] }{\mathbb{E}_{N,\beta}\left[W_{\gamma_1 + \gamma_2 }\right]\mathbb{E}_{N,\beta}\left[W_{0} \right]}
        =
        \frac{\sum_{{n}_1 \in \mathcal{C}_{\gamma_1},\, {n}_2 \in \mathcal{C}_{\gamma_2}} w({n}_1)w({n}_2)}{\sum_{n_1 \in \mathcal{C}^+_0,\, {n}_2 \in \mathcal{C}_{\gamma_1 + \gamma_2}} w({n}_1)w({n}_2)}
        \\&\qquad
        =
        \frac{\sum_{{n}_1 \in \mathcal{C}_{0},\, {n}_2 \in \mathcal{C}_{\gamma_1 + \gamma_2}} w({n}_1)w({n}_2) \mathbb{1}(\exists q \in \mathcal{C}_{\gamma_1} \colon q \leq n_1 +n_2 )}{\sum_{n_1 \in \mathcal{C}_0,\, {n}_2 \in \mathcal{C}_{\gamma_1 + \gamma_2}} w({n}_1)w({n}_2)}. 
    \end{align*}
    Since the latter term is a probability, it follows that it is smaller than or equal to one, and hence
    \begin{equation*}
        \frac{\mathbb{E}_{N,\beta}\left[W_{\gamma_1} \right]\mathbb{E}_{N,\beta} \left[W_{\gamma_2} \right] }{\mathbb{E}_{N,\beta}\left[W_{\gamma_1} W_{\gamma_2} \right]}  \leq 1.
    \end{equation*} 
    Using Proposition~\ref{proposition: positivity} and rearranging, we obtain~\ref{item: griffiths}.

    To see that~\ref{item: monotonicity} holds, note that
    \begin{align*}
        &\frac{d}{d\beta} \mathbb{E}_{N,\beta}\left[W_{\gamma} \right]
        =
        \beta \sum_{p \in C_2(B_N)} \bigl( \mathbb{E}_{N,\beta}\left[ W_{\gamma + \partial p} \right] - \mathbb{E}_{N,\beta}\left[ W_\gamma \right]  \mathbb{E}_{N,\beta} \left[ W_{\partial p} \right] \bigr)
    \end{align*}
    Using~\ref{item: griffiths} with \( \gamma_1 = \gamma\) and \( \gamma_2 = \partial p\) for each \( p \in C_2(B_N),\) the desired conclusion immediately follows.
\end{proof}
 \begin{corollary}
 Let $\gamma$ be an axis-parallel rectangle with side-lengths $R,T$. Then for any $\beta$, the quark potential
 \[
 V_\beta(R) = -\lim_{T \to \infty} \frac{1}{T} \log \, \langle W_{\gamma_{R,T}} \rangle_\beta
 \]
 exists.
 \end{corollary}
\begin{proof}
    By part (i) of Proposition~\ref{proposition: griffiths}, $T \mapsto -\log \langle W_{\gamma_{R,T}} \rangle_\beta$ is subadditive. Applying Fekete's lemma, we obtain the desired conclusion.
\end{proof}

Given a loop \( \gamma,\) note that 
\[
\area (\gamma) = \min_{n \in \mathcal{C}_\gamma} \sum_{p \in C_2(B_N)} n(p).
\]
We now prove that the Wilson loop satisfies an area law estimate when \( \beta\) is sufficiently small. (This is the only place in this note where a restriction on $\beta$ is needed.) It would be interesting to try to use the current expansion to study Wilson loop expectations in other ranges of $\beta$, especially near the critical value where the phase transition occurs, see, e.g., \cite{FV} and the references therein. 

\begin{proposition}\label{proposition: area law for small beta}
    Let \( \beta \geq 0,\) let \( \gamma \) be a loop, and let \( N \) be large enough to ensure that \( \dist(\gamma,\partial B_N) \geq \area(\gamma). \) Further, assume that \(   \beta  <1/\bigl( 4(m-1)\bigr). \)  Then for all $R \ge 1$,
    \[
    V_\beta(R) \ge a_{\beta}R, \quad a_{\beta} = \log \frac{1}{4(m-1)\beta}.
    \]
    
\end{proposition}

\begin{proof} 
We will prove the stronger estimate
\begin{equation*}
        \mathbb{E}_{N,\beta}\left[ W_\gamma  \right] \leq \frac{(4(m-1)\beta)^{\area(\gamma)}}{1 - 4(m-1)\beta}.
    \end{equation*}
    For each \( j \geq 0, \) we will now construct a set of currents \( \mathcal{C}_{\gamma,j} \subseteq \mathcal{C}_{\gamma}\)  whose supports induces connected subgraphs of \( \mathcal{G} \) as follows.
    Fix any ordering of the edges in \( C_1(B_N)^+. \) Note that \( |\hat \partial e| \leq 2(m-1) \) for any edge \( e \in C_1(B_N), \) and that any  \( n \in \mathcal{C}_\gamma\) must satisfy 
    \[
        \sum_{p \in C_2(B_N)^+} n(p) \geq \area(\gamma).
    \]
    
    Let \( n_0 = 0.\)
    For each \( j \geq 0,\) if \( n_j \in \mathcal{C}_\gamma,\) then we let \( n_{j+1}= n_j.\) If not, we construct \( n_{j+1}\) as follows:
    \begin{enumerate}
        \item Let \( e_j \in C_1(B_N)^+\) be the first edge (with respect to the ordering of the edges) for which 
        \[ \gamma(e_j)  + \sum_{p \in \support \hat \partial e_j} n(p)  \not \equiv 0 \mod 2.\]
        Since \( n_j \notin \mathcal{C}_\gamma\) by assumption, the edge \( e_j \) is well defined. 
        \item Pick any of the \( 2(m-1) \) plaquettes in \( \support \hat \partial e_j,\) and call this plaquette \( p_j. \)
        \item For \( p \in C_2(B_N),\)  define \( n_{j+1}(p) \coloneqq \begin{cases}
            n_j(p) &\text{if } p \neq p_j \cr 
            n_j(p)+1 &\text{else}.
        \end{cases}\)
    \end{enumerate}
    For \( j \geq 0,\) we let \( \mathcal{C}_{\gamma,j}\) be the set of all currents \( n_j\) that arise in this way that satisfies \( n_j \in \mathcal{C}_{\gamma} \) and \( \sum_{p \in C_2(B_N)^+} n(p)  = j.\) By construction, we have \( |\mathcal{C}_{\gamma,j}| \leq \bigl(2(m-1)\bigr)^j.\)
    For \( n \in \mathcal{C}_\gamma,\) we can write \( n = n^\gamma + (n-n^\gamma,)\) where \( n^\gamma \in \mathcal{C}_\gamma\) and \( n-n^\gamma \in \mathcal{C}_0\) have disjoint supports. 
    Combining the above observations, we obtain
    \begin{equation*}
        \begin{split}
        & \mathbb{E}_{N,\beta}\left[ W_\gamma  \right]
        =
        \frac{
            \sum_{{n} \in \mathcal{C}_\gamma} 
            w({n}) 
        }{
            \sum_{{n} \in \mathcal{C}_0}  
            w({n}) 
        }
        = 
        \frac{
            \sum_{{n} \in \mathcal{C}_\gamma} 
            w({n^\gamma}) w(n-n^\gamma) 
        }{
            \sum_{{n} \in \mathcal{C}_0}  
            w({n}) 
        } 
        \leq  
        \frac{ 
            \sum_{j \geq 1} 
            \sum_{n_j\in \mathcal{C}_{\gamma,j}}
            \sum_{n\in \mathcal{C}_0} 
            w(n_j)w(n) 
        }{
            \sum_{{n} \in \mathcal{C}_0}  
            w({n}) 
        }
        \\&\qquad \leq
        \sum_{j \geq 1}  \sum_{n_j\in \mathcal{C}_{\gamma,j}} w(n_j) 
        \leq 
        \sum_{j \geq \area(\gamma)} \bigl(2(m-1)\bigr)^j (2\beta)^j.
        \end{split}
    \end{equation*}
    From this the desired conclusion immediately follows. 
\end{proof}

In the following result we use~Theorem~\ref{proposition: couplings of many models} to establish stochastic domination of the random cluster model by Bernoulli plaquette percolation. The lower bound is due to~\cite[Lemma~32]{ds2023}, which then uses it to show that there are two distinct phases with area and perimeter law, respectively (see also~\cite{accfr}).

\begin{proposition}\label{prop: stoch dom}
	Let \( \beta \geq 0, \) and for \( p \in (0,1), \) let \( \psi_p \) denote the law of iid Bernoulli plaquette percolation. Then
	\begin{align*}
		\Psi_{\tanh 2\beta} \leq \phi_{B_N,\beta}^0  \leq \Psi_{1-e^{-4\beta}},
	\end{align*} 
	where \( \leq \) is stochastic domination.
\end{proposition}

\begin{proof}
	First note that the lower bound is an immediate consequence of Theorem~\ref{proposition: couplings of many models}.
	
	For the upper bound, analogously to the proof of~\cite[Lemma~32]{ds2023}, we note that for an  \( P \in \mathcal{P}_0 = \Omega_2(B_N,\mathbb{Z}_2), \) we have 
\begin{equation*}
		\phi_{B_N,\beta}^0 (P) \coloneqq 
		Z^{-1} 2^{\mathbf{b}_2 (P,\mathbb{Z}_2)}\prod_{p \in P} (1-e^{-4\beta_p}) \prod_{p\notin P}  e^{-4\beta_p} .
	\end{equation*}	
	Let \( p_0 \in C_2(B_N) \). Then
	\begin{align*}
		\mathbf{b}_2 (P\cup \{p_0\},\mathbb{Z}_2) \in \pigl\{ \mathbf{b}_2 (P\smallsetminus \{ p_0\} ,\mathbb{Z}_2) , \mathbf{b}_2 (P\smallsetminus \{ p_0\} ,\mathbb{Z}_2) -1 \pigr\}.
	\end{align*}
	From this, it follows that
	\begin{align*}
		\frac{e^{-4\beta}}{1-e^{-4\beta}} \leq \phi_{B_N,\beta}^0 (P \smallsetminus \{ p_0\})/\phi_{B_N,\beta}^0 (P \cup \{ p_0\}) \leq \frac{2e^{-4\beta}}{1-e^{-4\beta}} 
	\end{align*}
	and hence
	\begin{align*}
		&\frac{\phi_{B_N,\beta}^0 (P \cup \{ p_0\})}{\phi_{B_N,\beta}^0 (P \cup \{ p_0\})+\phi_{B_N,\beta}^0 (P \smallsetminus \{ p_0\})}
		=
		\frac{1}{1+\phi_{B_N,\beta}^0 (P \smallsetminus \{ p_0\})/\phi_{B_N,\beta}^0 (P \cup \{ p_0\})}
		 \\&\qquad \in \Bigl( \frac{1}{1+\frac{2e^{-4\beta}}{1-e^{-4\beta}} },\frac{1}{1+\frac{e^{-4\beta}}{1-e^{-4\beta}} } \Bigr) = 
		 ( \tanh 2\beta, 1-e^{-4\beta}  ).
	\end{align*}
	This concludes the proof.
\end{proof}

The following result is an analog of~Proposition~\ref{prop: stoch dom} for the random current model.

\begin{proposition}\label{prop: stoch dom ii}
		Let \( \beta \geq 0, \) and for \( p \in (0,1), \) let \( \Psi_p \) denote the law of iid Bernoulli plaquette percolation. Then
	\begin{align*}
		\Psi_{1-1/\cosh 2\beta} \leq  \hat{\mathbf{P}}^0_{B_N,\beta} \leq \Psi_{1-e^{-4\beta}}
	\end{align*}
\end{proposition}

\begin{proof} 
	The lower bound follows immediately from noting that for any plaquette \( p \in C_2(B_N)\), we have
	\begin{align*}
		\hat{\mathbf{P}}^0_{B_N,\beta} \bigl(\hat {{n}}(p) \geq 0 \bigr) = {P}^0_{B_N,\beta} \times \Psi_{1-1/\cosh 2\beta} \bigl( \eta(p) + X_1(p) \geq 0 \bigr) \geq \Psi_{1-1/\cosh 2\beta}(X_1(p) \geq 0).
	\end{align*}
	The upper bound follows by combining Proposition~\ref{prop: stoch dom} and Theorem~\ref{proposition: couplings of many models}.
\end{proof}

In the final result of this section, Proposition~\ref{prop: correlations ii} below, we show that if \( \beta \) is sufficiently large or sufficiently small, then Ising lattice gauge theory has exponential decay of correlations.

\begin{proposition}\label{prop: correlations ii}
	If \( \beta>0 \) is either sufficiently large or sufficiently small, then there are \( C_\beta,c_\beta>0 \) such that for any disjoint loops \( \gamma\) and \( \gamma' \), we have
	\begin{align*}
		\Cov (W_{\gamma},W_{\gamma'}) \leq C_\beta e^{-c_\beta \dist(\gamma,\gamma')}.
	\end{align*}
\end{proposition}

\begin{proof} 
	By using the current expansion and Lemma~\ref{lemma: the switching lemma} (the switching lemma), we have
	\begin{align*}
		&\frac{\Cov (W_{\gamma},W_{\gamma'})}{\mathbb{E}[W_{\gamma+\gamma'}]} 
		= 
		\mathbf{P}^{0,\gamma+\gamma'}_{B_N,\beta}( {n}_1 + {n}_2 \nleq \mathcal{C}_{\gamma}) .
		%\\&\qquad\leq
		%P^{\gamma+\gamma'}_{B_N,\beta}( \eta \nleq \mathcal{C}_{\gamma})   
	\end{align*} 
	%If \( \eta \sim P^{\gamma+\gamma'}_{B_N,\beta} \) and \( \eta \nleq \mathcal{C}_\gamma, \) then \( \eta \) must contain a "tube" with boundary \( \gamma+\gamma'.\) 
	Now note that if \( \partial  {n}_1 \in \mathcal{C}_0 \), \( \partial  {n}_2  \in \mathcal{C}_{\gamma+\gamma'}, \) and \( {n}_1 + {n}_2 \nleq \mathcal{C}_{\gamma} , \) then the following events must both occur.
	\begin{enumerate}
		\item \( \mathcal{E}_2 \): \( \support \hat n_2 \) contains a connected component of plaquettes which is adjacent to both \( \gamma \) and \( \gamma', \) and hence have size at least \( \dist (\gamma,\gamma'). \)
		\item \( \mathcal{E}_1 \): \( \support \hat n_1 \) does not contain a set of plaquettes that is adjacent to every connected set of plaquettes adjacent to both \( \gamma \) and \( \gamma'. \) 
	\end{enumerate}
	Hence
	\begin{align*}
		&\mathbf{P}^{0,\gamma+\gamma'}_{B_N,\beta}( {n}_1 + {n}_2 \nleq \mathcal{C}_{\gamma})  
		\leq \mathbf{P}^{0}_{B_N,\beta}( \mathcal{E}_1)   \mathbf{P}^{\gamma+\gamma'}_{B_N,\beta} \bigl( \mathcal{E}_2\bigr).   
	\end{align*}  
	We will deal with the probabilities of the events \( \mathcal{E}_1 \) and \( \mathcal{E}_2 \)  separately. For \(\mathcal{E}_2 \), let \( q \) and \( q' \) be two surfaces with boundaries \( \gamma \) and \( \gamma' \) respectively, and note that \( \eta \mapsto \eta+q+q' \) is a bijective map between \( \mathcal{P}_0^{ht} \) and \( \mathcal{P}_{\gamma+\gamma'}^{ht}. \) To simplify notation, let \( p_2 \coloneqq p_2(\beta) \coloneqq 1-1/\cosh 2\beta \)  and \( p_2 \coloneqq p_1(\beta) \coloneqq \tanh 2\beta .\)  
	Then, letting \( X_2 \sim \Psi_{p_1} \) and using Theorem~\ref{proposition: couplings of many models} and Proposition~\ref{prop: stoch dom ii}, we get 
	\begin{align*}
		&%\mathbf{P}^{0,\gamma+\gamma'}_{B_N,\beta}( {n}_1 + {n}_2 \nleq \mathcal{C}_{\gamma})  
		%\leq  
		\mathbf{P}^{\gamma+\gamma'}_{B_N,\beta} \bigl( \mathcal{E}_2\bigr) 
		=
		{\hat {\mathbf{P}}}^{\gamma+\gamma'}_{B_N,\beta} \bigl( \hat n \in \mathcal{E}_2\bigr)  
		=
		\hat{\mathbf{P}}^{\gamma+\gamma'}_{B_N,\beta} \bigl( \hat n \in \mathcal{E}_2\bigr)  
		=
		{P}^{\gamma+\gamma'}_{B_N,\beta} \times \Psi_{p_1}\bigl( \max(\eta,X_1) \in \mathcal{E}_2\bigr)
		\\&\qquad
		\leq
		p_1^{-(|q|+|q'|)} {P}^{0}_{B_N,\beta} \times \Psi_{p_1}\bigl( \max(\eta,X_1) \in \mathcal{E}_2\bigr)
		=
		p_1^{-(|q|+|q'|)} \hat{\mathbf{P}}^{0}_{B_N,\beta} \bigl( \hat n \in \mathcal{E}_2\bigr)  
		\\&\qquad=
		p_1^{-(|q|+|q'|)} {\mathbf{P}}^{0}_{B_N,\beta} \bigl(  \mathcal{E}_2\bigr)   
		\leq
		p_1^{-(|q|+|q'|)} \Psi_{1-e^{-4\beta}}\bigl(  \mathcal{E}_2\bigr)  . 
	\end{align*} 
	This implies the desired conclusion in the case \( \beta \) is small.
	On the other hand,  Proposition~\ref{prop: stoch dom ii}	 immediately implies that
	\begin{align*}
		&\mathbf{P}^{0}_{B_N,\beta} \bigl( \mathcal{E}_1\bigr)  
		\leq
		\Psi_{1-e^{-4\beta}}\bigl(  \mathcal{E}_2\bigr), 
	\end{align*}
	and hence if \( \beta \) is large, then \( \mathbf{P}^{0}_{B_N,\beta} \bigl( \mathcal{E}_1\bigr) \)  has exponential decay in \( \dist(\gamma,\gamma') \)) (see, e.g., \cite[Theorem 1]{gh2010}). This concludes the proof.
 \end{proof}

\end{document}